\documentclass[11pt]{amsart}
\usepackage{amssymb}
\usepackage{amsmath}

\usepackage[latin1]{inputenc}
\usepackage{enumerate}

\newtheorem*{acorollary}{Corollary}

\usepackage{color}

\definecolor{orange}{rgb}{1,0.5,0}

\def\Z{{\mathbb Z}}\def\T{{\mathbb T}}\def\R{{\mathbb R}}\def\C{{\mathbb C}}

\def\e{\eta}
\def\g{\gamma}
\def\s{\sigma}
\def\al{\alpha}\def\be{\beta}
\def\be{\beta}

\def\b#1{\lbrace#1\rbrace}
\def\a#1{\left|#1\right|}

\def\<{\langle}
\def\>{\rangle}

\theoremstyle{plain}

\def\a{\alpha}

\def\ti{\tilde}
\def\e{\varepsilon}

\def\s{\sigma}

\let\newpf\proof \let\proof\relax

\def\cL{\mathcal{L}}

\newcommand{\ba}{\overline{A}}

\newcommand{\cF}{\mathcal{F}}

\newcommand{\cO}{\mathcal{O}}

\def\be{\begin{equation}}
\def\ee{\end{equation}}

\def\ba{{\begin{align}}}
\def\ea{{\end{align}}}

\def\bm{\begin{pmatrix}}
\def\em{\end{pmatrix}}

\def\a{{\alpha}}

\def\g{{\gamma}}

\def\0{{\mathbf 0}}

\newtheorem{theo}{Theorem}

\newtheorem{prop}{Proposition}[section]

\theoremstyle{remark}

\theoremstyle{definition}

\newenvironment{proof}{ \noindent{\it Proof.}\quad}{\ \hfill $\Box$\vskip .2cm}

\def\ssm{\smallsetminus}

\renewcommand{\setminus}{\ssm}

\newcommand{\id}{\operatorname{id}}

\newcommand{\eps}{{\epsilon}}

\newcommand{\N}{{\mathbb N}}
\newcommand{\Q}{{\mathbb Q}}

\def\B0{{\bold{0}}}
\def\b{\beta}
\def\a{\alpha}

\def\be{\begin{equation}}
\def\ee{\end{equation}}

\def\cB{\mathcal{B}}

\def\cF{\mathcal{F}}

\def\cH{\mathcal{H}}

\def\1{{\bf 1}}

\newcommand\Diff{{\rm Diff}}
\newcommand\dD{{\rm D}}

\catcode`\@=12

\def\Empty{}
\newcommand\oplabel[1]{
  \def\OpArg{#1} \ifx \OpArg\Empty {} \else
  	\label{#1}
  \fi}
		
\newcommand{\comm}[1]{}
\newcommand{\comment}[1]{}

\begin{document}
\title[J.-C. Yoccoz and the theory of circle diffeomorphisms]{Jean-Christophe Yoccoz and the theory of circle diffeomorphisms}
\author{Hakan Eliasson, Bassam Fayad and  Raphaël Krikorian}

\email{hakan.eliasson@math.jussieu.fr} 
\address{Bassam Fayad, IMJ-PRG CNRS}
\email{bassam@math.jussieu.fr} 
\address{Raphaël Krikorian, Universit\'{e} de Cergy-Pontoise} 
\email{raphael.krikorian@u-cergy.fr}

\maketitle

{\it  
Pour beaucoup d'entre nous, Jean-Christophe Yoccoz était à la fois un frère, un ami, un modèle et un point de repère. 

C'était une personne intègre, attachante, contrastée. Son caractère tempéré cachait tant d'extrêmes : talent exceptionnel et étonnante simplicité, rigueur et légèreté,  dévouement et intransigeance, joie de vivre profonde et absence de toute frivolité.

Mathématicien de grand renom, il acceptait, sans en tirer aucune vanité, le rôle important qu'il avait à jouer dans le monde scientifique et la communauté mathématique en particulier.  

Jean-Christophe aimait transmettre et partager sa passion des mathématiques. Il trouvait toujours du temps pour les collègues, surtout les jeunes, qu'il écoutait avec enthousiasme et attention, d'un oeil  critique mais jamais acerbe.

On se souviendra de lui toujours souriant, entouré de ses amis et collègues, joyeux aussi bien pour une belle idée nouvelle que pour un rayon de soleil.

}

\bigskip

A crucial discovery made by Kolmogorov around the middle of last century is that  
 most (in a measure theoretical sense) quasi-periodic motions are robust 
  in a number of physical situations, including the quasi-integrable Hamiltonian systems. Indeed, KAM Theory (after Kolmogorov-Arnold-Moser) established the persistence under perturbations  of most of the invariant  tori of  an integrable Hamiltonian system.

Prior to that, Poincar\'e and Denjoy's theory of circle homeomorphisms established that any  $C^2$ circle diffeomorphism that does not have periodic orbits is topologically conjugate to an irrational rotation, and is thus quasi-periodic. With the subsequent works of many great mathematicians, the theory of quasi-periodic motions on the circle grew to become one of the most complete and inspiring theories of modern dynamics.

 Our aim in this note is to present some of the the crucial contributions of Jean-Christophe Yoccoz in this field. We start with a short historical account before exposing Yoccoz' work. Then we give a brief description of the main conceptual and technical tools of the theory, with a focus on describing Yoccoz' work and contributions.

\section{Circle diffeomorphisms}

\subsection{Homeomorphisms of the circle. Poincaré theory}
Let $r\in\N\cup\{\infty,\omega\}$. We denote by $\Diff_{+}^r(\T)$ the group (for the composition) of orientation preserving homeomorphisms of $\T=\R/\Z$ which, together with their inverses,  are of class $C^r$. We shall be particularly concerned with the {\it analytic} ($r=\omega$) and the {\it smooth} ($r=\infty$) cases. By $T_\alpha$, $\alpha\in\T$,  we shall  denote the rotation of $\T$ by $\alpha$, $x\mapsto x+\a$ ---  it is an element of $\Diff_{+}^r(\T)$. 

We denote by $\dD^r(\T)$ its universal covering space which is the set of $C^r$-diffeomorphims of the real axis $\R$ commuting with the translation by 1 $T:x\mapsto x+1$. More generally, when $\alpha\in\R$, $T_\a$ will denote the translation of $\R$ by $\alpha$  ---  it is an element of  $\dD^r(\T)$. We have $\Diff^r_{+}(\T)\simeq \dD^r(\T)/\{T^p,p\in\Z\}$ and we denote by $\pi$ the canonical projection $\dD^r(\T)\to \Diff^r_{+}(\T)$.

Let $\T_\nu=\{z\in \C: \vert \Im z\vert\leq \nu\}/\Z$ and let $\Diff_{+}^\omega(\T_\nu)$ denote the subgroup of real analytic  diffeomorphismes on $\T$ that extends holomorphically to $\T_\nu\to\C/\Z$. 
For $f,g\in\Diff_{+}^\omega(\T_\nu)$ let
$$\vert\vert f-g\vert\vert_{\nu}=\sup_{z\in\T_\nu}\vert\vert f(z)-g(z)\vert\vert$$
(where $\vert\vert\ .\ \vert\vert$ denotes the metric on $\C/\Z$).

\subsubsection*{The rotation number} Poincaré  laid the foundations of the theory of circle dynamics. In particular, he introduced the central notion of rotation number and began the classification up to conjugation of these systems.  The rotation number $\rho(\bar f)\in\R$ of  a homeomorphism  $\bar f\in \dD^0(\T)$ is the uniform limit 
$$\lim_{n\to\infty}\frac{{\bar f}^n(x)-x}{n}$$
(this limit exists and is independent of $x\in\T$).
If   $f\in \Diff^r_{+}(\T)$ and $\pi(\bar f)=f$, $\bar f\in \dD^r(\T)$  the rotation number $\rho(f)$ of $f$ is   the element of $\R/\Z$,  $\rho(f)=\rho(\bar f)\mod 1$; this element of $\T$ is independent of the choice of the lift $\bar f$ of $f$. If $\mu$ is an invariant probability measure for $f$ one has the following relation
\begin{equation} \label{eq.rot} \rho(\bar f)=\int_{\T} (\bar f(x)-x)d\mu(x). \end{equation}

The rotation number has many nice properties; let's mention some  of them: the rotation number of a rotation $T_{\a}$ is equal to $\a$,  the rotation number $\rho(f)$ depends continuously on $f$ and it  is invariant under {\it  conjugation}  which means that if $f,g\in \Diff_{+}^0(\T)$ satisfy $f\circ h=h\circ g$ where $h:\T\to\T$ is an orientation preserving  homeomorphism then $\rho(f)=\rho(g)$.

\subsubsection*{Poincaré theorem}
The rotation number of a rotation $T_{\a}$ is equal to $\a$ and it is thus natural to ask whether translations are a universal model for circle homeomorphisms, but this is not always  the case. Indeed, it is true that  a homeomorphism of the circle has rational rotation number if and only if it has at least one periodic point, but unlike the case of translation with rational rotation number, the set of periodic points of  such a homeomorphism could be  at most countable. On the other hand when the rotation number of a homeomorphism is irrational one has more structure:
\begin{theo}[Poincaré]\label{theo:Poincare} Let $f$ be a homeomorphism of the circle with irrational rotation number. Then, there exists $h:\T\to\T$ surjective, orientation preserving and continuous such that 
$$ h\circ f=T_{\rho(f)}\circ h.$$

In other words $(\T,T_{\rho(f)})$ is a topological  factor of $(\T,f)$.
\end{theo}

{ Poincaré's theorem and its proof have many fundamental consequences. In particular one can prove that any orientation preserving homeomorphism of the circle with irrational rotation number is uniquely ergodic, i.e. has a unique invariant probability measure. (If $\mu_{f}$ is the invariant probability measure of $f$, one can choose $h(x)=\mu_{f}([0,x])$ if $x>0$ and $h(x)=-\mu_{f}([0,x])$ if $x<0$). The map $h$ is a homeomorphism if and only if the topological support of $\mu_{f}$ is the whole circle. If not, the intervals on which $h$ is not strictly increasing belong to orbits of wandering intervals of $f$ \footnote{A wandering interval is an interval disjoint from its image by any iterate of $f$.} which accumulate by unique ergodicity on the support of $\mu_{f}$, which is the minimal set of $f$ \footnote{A minimal set of $f$ is a nonempty, closed, invariant set which is minimal (for the inclusion) for these properties.}. 
On its minimal set $f$ is isomorphic  to $R_{\rho(f)}$.}

\subsection{Diffeomorphisms of the circle. Denjoy theory.}
In view of Poincaré theorem it is natural to ask for conditions ensuring {\it $C^0$-linearization}, i.e.  that the semi-conjugation $h$ defined above is in fact a homeomorphism. This question was answered by Denjoy in the 30's. Denjoy constructed examples of diffeomorphisms $f$ of class $C^1$ that are not conjugated to rotations (i.e. such that the support of $\mu_{f}$ is not the whole circle). But he also showed that this cannot happen when $f$ is of class $C^2$.

\begin{theo}[Denjoy \cite{De}]\label{theo:denjoy} Le $f$ be an orientation preserving diffeomorphism of the circle with an  irrational rotation number such that $\ln Df$ has bounded variation (for example $f$ is of class $C^2$). Then there exists a homeomorphism $h:\T\to\T$ such that
\be h\circ f=T_{\rho(f)}\circ h.
\ee
\end{theo}
Notice that the above {\it conjugacy} $h$ is almost unique in the sense that if $\ti h$ is another homeomorphism conjugating $f$ to $T_{\rho(f)}$  then $\ti h=T_{\b}\circ h$ for some $\b\in\T$.

\subsection{$C^r$- linearization: the local theory}

The next natural step is to ask for conditions ensuring {\it $C^r$-linearization}, i.e.  $C^r$-regularity for the conjugation(s) $h$ given by Denjoy theorem. Already, to get a continuous conjugation we saw that one has to assume  that the rotation number is irrational and that the map $f$ is not only a diffeomorphism, but also  has higher regularity (for example $C^2$). However, higher regularity of  $f$ is not sufficient as was shown by some examples of  Arnold.  One also needs   {\it arithmetic conditions} on the rotation number.

An irrational  number $\a\in(0,1)$ is said to be {\it Diophantine} with {\it exponent} $\sigma\geq 0$ and {\it constant} $\gamma>0$ if it satisfies 
$$ \forall  (p,q)\in\Z\times \Z^*,\ |\a-\frac{p}{q}|\geq \frac{\gamma}{q^{2+\sigma}}. 
$$
Notice that this is indeed a condition on the angles $\al\in\T$, and we 
denote the set of  such $\a$ by $\mathcal{CD}(\gamma,\sigma)$.  As soon as $\sigma>0$ the Lebesgue measure of the set $\T\setminus \mathcal{CD}(\gamma,\sigma)$ goes to zero as $\gamma$ goes to zero and the union $\mathcal{CD}(\sigma):=\bigcup_{\gamma>0}\mathcal{CD}(\gamma,\sigma)$ has full Lebesgue measure in $\T$. An angle is {\it Diophantine} if it belongs to 
$\mathcal{C}\mathcal{D}=\bigcup_{\sigma\geq 0} \mathcal{CD}(\sigma)$. An irrational angle that is not Diophantine is called Liouville.

\subsubsection*{Arnold theorem, KAM theory}
Under such an arithmetic condition on the rotation number  Arnold proved the first $C^{\omega}$-linearization result.

\begin{theo}[Arnold \cite{Arnold}]\label{theo:arnold}  For any $\a\in {\mathcal{CD}}$ and $\nu>0$ there exists $\e(\a,\nu)>0$ with the following property: 
any analytic  diffeomorphism $f\in \Diff^\omega_{+}(\T_\nu)$ such that
$\rho(f)=\a$ and $\|f-T_{\rho(f)}\|_{\nu}<\e(\rho(f),\nu)$ 
 is conjugated to $T_{\rho(f)}$ by an analytic  diffeomorphism $h$ in $ \Diff^\omega_{+}(\T_{\frac\nu2})$. 
 
Moreover, for any $\a\in \mathcal{CD}(\gamma,\sigma)$, there exists such an $\e(\a,\nu)$ that only depends on $\sigma,\gamma$ and $\nu$.
\end{theo}
The proof of this result  was  one of the first application of the new  ideas designed by Kolmogorov and Arnold (and later by  Moser in the smooth, that is $C^\infty$,  case) to handle the so-called {\it small-divisors problems} (see Section \ref{sec:kam}). This would ulteriorly be encompassed in  what is now known as KAM theory. By essence, KAM theory is a perturbative method and requires crucially a closeness condition to the rotation. 

Later Michel Herman proved a smooth (and also a differentiable) version of this result. 
He also proved that the Diophantine condition was optimal in the smooth case. Indeed,   for any Liouville angle $\a$, it is easy to construct {\it  via} successive conjugations a smooth circle diffeomorphism with rotation number $\a$ that is not absolutely continuously conjugate to $T_\a$ (See for example \cite[Chapter XI]{He}). 

Herman knew that the Diophantine condition wasn't optimal in the analytic category\footnote{Herman cites in \cite{He} the two papers  of R\"ussmann
\cite{Ru1}, \cite {Ru2}. See also  the work of  Brjuno \cite{Br}.}, but it was Yoccoz who settled the question of  the optimal condition for  linearization to hold in the local context. 

An irrational  number $\a\in(0,1)$ always admits a unique {\it  continuous fraction expansion}
$$\a=1/(a_{1}+1/(a_{2}+\cdots)$$
where $a_{1},a_{2},\ldots$ are positive integers. We then denote $\a=[a_{1},a_{2},\ldots]$. 
The rational number  $p_{n}/q_{n}:=[a_{1},\ldots,a_{n}]$,  $p_{n},q_{n}\in\N^*$, ${\rm gcd}(p_{n},q_{n})=1$ is then   called  the  $n$-th  {\it convergent} of $\a$. 

The irrational number $\a$
 is said to verify the  {\it Brjuno condition} if and only if
$$ \sum_{n=1}^\infty\frac{\ln q_{n+1}}{q_{n}}<\infty,$$
where $p_{n}/q_{n}$ are the convergents of $\a$. This is also a condition on the angles $\al\in\T$, and  the set of  such $\a$ is usually denoted by $\mathcal{B}$.

\begin{theo}[Yoccoz \cite{Y-analyticdiffeo}]
\label{theo:yoccoz-analytic-local} For any $\a\in\mathcal{B}$ and $\nu>0$ there exists $\e(\a,\nu)>0$ with the following property: 
any analytic  diffeomorphism $f\in \Diff^\omega_{+}(\T_\nu)$ such that
$\rho(f)=\a$ and $\|f-T_{\rho(f)}\|_{\nu}<\e(\rho(f),\nu)$ 
 is conjugated to $T_{\rho(f)}$ by an analytic  diffeomorphism $h$ in $ \Diff^\omega_{+}(\T_{\frac\nu2})$. 
\end{theo}

Another ``local'' result obtained by Yoccoz gives analytic linearization under the Brjuno condition, when 
$f$ is injective on sufficiently large ``bands''. This is formulated in terms of the {\it Brjuno function} which is a $1$-periodic function $B:\R\to\R$ satisfying
 $$B(x)=xB(\{x^{-1}\})+\log x^{-1}, \quad x\in]0,1[.$$
 $B$ determines the Brjuno condition because $\a\in \mathcal B {\rm \ iff \ } B(\a)<+\infty$  (see \cite{Y-analyticdiffeo}). 

\begin{theo}[Yoccoz \cite{Y-analyticdiffeo}] \label{theo:yoccoz-analytic-local2}
 There exists $\nu_0>0$ such that if $f\in \Diff^\omega_{+}(\T)$, with $\rho(f) =\a \in {\mathcal B}$, and $f$  is analytic and injective on $\T_\nu$, with $\nu \geq{ \frac1{2\pi}} B(\a) +\nu_0$, then $f$
 is $C^{\omega}$-linearizable.
 \end{theo} 
 It is not hard to obtain Theorem \ref{theo:yoccoz-analytic-local} as a corollary of Theorem \ref{theo:yoccoz-analytic-local2}.

Combining the results of Yoccoz \cite{Y-germs} on  the optimality of  condition $\mathcal{B}$ for holomorphic germs and a theorem of Perez-Marco \cite{PM} that makes a bridge between holomorphic germs and analytic circle diffeomorphisms it can be seen that condition $\mathcal{B}$ is also optimal in Theorem \ref{theo:yoccoz-analytic-local}. 
Indeed, Yoccoz shows that

\begin{theo}[Yoccoz \cite{Y-analyticdiffeo}]\label{theo:yoccoz-analytic-counterexample} For any $\al\notin \mathcal{B}$, any $\nu>0$ and any $\e>0$, the following holds:
there exists an analytic  diffeomorphism $f\in \Diff^\omega_{+}(\T_\nu)$ such that
$\rho(f)=\a$ and $\|f-T_{\rho(f)}\|_{\nu}<\e$ that is not analytically linearizable. 
\end{theo}

\subsection{Smooth linearization}
It is one of the great  achievements of Herman to have  proved that one can in fact get a {\it global} ({\it i.e.} non perturbative) linearization result for smooth circle diffeomorphisms of the circle for {\it almost every} rotation number (this was conjectured by Arnold in the analytic category). The important contribution of Yoccoz was to prove that Herman's result extended to all Diophantine numbers (in the smooth case).

\subsubsection*{Herman-Yoccoz theorem}

\begin{theo}[Herman \cite{He}, Yoccoz \cite{Y-smoothdiffeo}]\label{theo:herman-yoccoz} Any $f\in \Diff^\infty_{+}(\T)$ with $\rho(f)\in \mathcal{CD}$ 
is $C^{\infty}$-linearizable.
\end{theo} 

It is clear that the Diophantine condition is optimal here since it is already optimal in the local case.

Yoccoz (as well as Herman) also proves a result
 for finitely differentiable diffeomorphisms.  In finite regularity, the conjugacy is less regular than the diffeomorphism: this phenomenon of loss of differentiability is typical of small divisors problems.

\begin{theo} [Herman \cite{He}, Yoccoz \cite{Y-smoothdiffeo}]\label{A} Any $f\in \Diff^r_{+}(\T)$ with $\rho(f)\in \mathcal{CD}(\s)\subset \mathcal{CD}$ 
and $\max(3,2\sigma+1)< r<\infty$ is $C^{r-\sigma-1-\eps}$-linearizable for any $\eps>0$. 
\end{theo}

{ Other versions of linearization of circle diffeomorphisms with alternative proofs were later obtained by Khanin and  Sinai \cite{KS} (for a subclass of full measure of the Diophantine numbers) and by  Katznelson and Ornstein \cite{KO1,KO2} (for all Diophantine numbers) that give the best known loss of differentiability. } 
 
\subsection{Analytic linearization  --  Yoccoz' renormalization theory}

The local theorem
 was first proven in the analytic category (Arnold's result, Theorem \ref{theo:arnold}). As is often the case in KAM theory, the analytic category is easier to handle compared to the finite differentiability case.
This is not the case in the global theory of circle diffeomorphisms. For instance, as was discovered by Yoccoz, the local condition for linearizability turns out to be weaker than the one necessary for the non-perturbative linearization theorem. 

The extension to the analytic case of Herman-Yoccoz global theorem does not follow from the smooth case and Yoccoz had to design a new approach  to attack this problem, inspired by his work on the optimality of the Brjuno condition for the linearization of holomorphic germs \cite{Y-germs}.  He found  the optimal arithmetic condition  on the rotation number, which he named  condition $\mathcal{H}$ in honor of  Herman,  that insures  the analytic linearization of analytic orientation diffeomorphisms of the circle.  Indeed, for an arbitrary number that does not satisfy condition $\mathcal H$, Yoccoz  constructs examples of real analytic diffeomorphisms with this rotation number that are not analytically linearizable. 

Thus, if we simply define  the set $\mathcal{H}$ as the set of all $\a\in\T$ such that any $f\in \Diff^\omega_{+}(\T)$ with rotation  number $\a$  is analytically linearizable, then Yoccoz gave a full description of this set. Namely, let us introduce the set $\cH$ as follows. Let 
$$ {\mathcal R}_{\a}(r):= \begin{cases}&\a^{-1}(r-\ln \a^{-1}+1) {\rm \ for \ } r \geq \ln \a^{-1} \\&e^r {\rm \ for \ } r \leq \ln \a^{-1} \end{cases} $$ 
Next, define inductively 
$$R_{n+1}(\a)={\mathcal R}_{\a_n}(R_n(\a)), \quad R_0(\a)=0$$
where $\a_n=G^n(\a)$ and $G(x):=\{1/x\}$ denotes the Gauss map. 
Then, a Brjuno number satisfies  Condition ${\mathcal H}$ if for any $m$, there exists $k\geq 0$ such that
$$R_k(\a_m)\geq B(\a_{m+k}),$$ 
where $B$ is the Brjuno function.   Yoccoz then showed

\medskip 

\begin{theo}[Yoccoz \cite{Y-analyticdiffeo}]\label{theo:yoccoz-analytic-global}
$$\mathcal{CD}\subset \mathcal{H}\subset \mathcal{B}$$
and both inclusions are strict.
\end{theo}
Yoccoz also gave a combinatorial description of $\mathcal{H}$ and showed in particular that it is an $F_{\s,\delta}$-set but not an $F_{\s}$-set (see \cite{Y-analyticdiffeo})\footnote{An $F_{\sigma}$-set is a countable union of closed sets. An $F_{\s,\delta}$-set is a countable intersection of $F_{\s}$-sets. }. 

{

In his 1994 ICM lecture \cite{Y-icm}, Yoccoz gives some examples that illustrate the difference between $\mathcal{CD}, \mathcal{H}$ and  
$\mathcal{B}$. If $\{p_n/q_n\}$ is the sequence of  convergents of $\a$, then $\a$ is Diophantine if, and only if,
$$\log(q_{n+1})=\mathcal{O}(\log q_n)\quad \forall n,$$
while any $\a$ with
$$\log(q_{n+1})=\mathcal{O}((\log q_n)^c)\quad \forall n,$$
for some $c>0$, belongs to $\mathcal{H}$. If $\{a_n\}$ are the coefficients in the continued fraction expansion of $\a$,
then  any $\a$ with
$$e^{(a_n)^c}\leq a_{n+1}\leq e^{a_n}\quad \forall n,$$
for some $0<c<1$, is always  Brjuno but never in $\mathcal{H}$.
}

This theorem implies in particular that there are angles $\a\in\mathcal{B}$ and
analytic orientation preserving circle diffeomorphisms $f$ with the rotation  number $\a$  that are not analytically linearizable.

\subsection{Beyond linearization. }

\subsubsection*{Density of linearization}
For $\a\in \T$, let $ \mathcal{F}_{\a}^\infty$ be the set of diffeomorphisms in $\Diff^\infty_{+}(\T)$ with rotation number $\a$, and let
$\mathcal{O}_{\a}^\infty\subset \cF_{\a}^\infty$ be the subset of $C^\infty$-linearizable diffeomorphisms.

The content of Herman-Yoccoz theorem is that if $\a$ is Diophantine, then $\cO_{\a}^\infty=\cF_{\a}^\infty$. 
What can be said for Liouville $\a$'s?

Herman proved that if
$\a$ is Liouville, then $\cO_{\a}^\infty$ is meager in $\cF_{\a}^\infty$ (see \cite[Chapter XI]{He}).

\begin{theo}[Yoccoz \cite{Y}]\label{theo:yoccoz-fo}For any Liouville $\a\in\T$, $\cF_{\a}^\infty=\overline{\cO_{\a}^\infty}$ (where the closure is for the $C^\infty$-topology).
\end{theo}
In other words, any smooth orientation preserving diffeomorphism of the circle can be approximated in the smooth topology by smoothly linearizable ones.

The immediate  consequence of  Theorem \ref{theo:yoccoz-fo}  that 
a property that is dense in the $C^\infty$ closure $\overline{\cO^\infty_\a}$ is actually dense in $\cF_\a^\infty$, plays a crucial role in understanding the  generic properties of diffeomorphisms in $\cF_\a^\infty$. 
Indeed, many examples of "exotic" (far from linearizable) behaviors that appear for quasi-periodic  systems with  Liouville frequencies are built in the class of diffeomorphisms that are limits of conjugates to periodic translations. {The density of $\cO^\infty_\a$ in $\cF_\a^\infty$ then permits to show that the observed behavior is dense or even generic in all of $\cF_\a^\infty$}. An example is the  Theorem \label{theo:13} below on the centralizer of a generic diffeomorphism of $\cF_\a^\infty$, $\a$ Liouville. 

\subsubsection*{Centralizers}
Let $Z^\infty(f)$ be the $C^\infty$-centralizer of $f\in\Diff^\infty_{+}(\T)$, i.e. $Z^\infty(f)=\{g\in\Diff_{+}^\infty(\T), f\circ g=g\circ f\}$ and $Z^\infty_{0}(f)$ the closure for the $C^\infty$-topology of the group of iterates of $f$; obviously $Z_{0}^\infty(f)$ is a closed subgroup of $Z^\infty(f)$.  By Herman-Yoccoz Theorem, when $\a\in{\mathcal C}{\mathcal D}$, for {\it all} $f\in F_{\a}^\infty$, $Z_{0}^\infty(f)=Z^\infty(f)$ and is indeed equal to $\{h\circ T_{\beta}\circ h^{-1}, \beta \in \T\}$ where $h$ is the linearizing diffeomorphism of $f$; in particular it is uncountable. When $\a$ is Liouville, Yoccoz proves the following theorem
\begin{theo}[Yoccoz \cite{Y}]\label{uncountable}
For any Liouville  $\a$, the generic diffeomorphism $f$ in $\cF_\a^\infty$ is such that  $Z^\infty_{0}(f)$ is uncountable. 
\end{theo}

 Yoccoz has also obtained other important results on centralizers for diffeomorphisms in
$$\cF_{I}^\infty=\bigcup_{\a\notin\Q} F^\infty_{\a}.$$
\begin{theo}[Yoccoz \cite{Y}]\label{theo:13} For a  generic set of  $f\in F^\infty_{I}$ one has 
$$Z^\infty(f)=Z^\infty_{0}(f).$$
 Moreover for a dense set of $f\in F^\infty_{I}$ one has 
 $$Z^\infty(f)=Z^\infty_{0}(f)=\{f^n, n\in\Z\}.$$  
 \end{theo}

\bigskip 

\medskip

\section{Conceptual and technical  tools}
We describe in this section some important concepts and tools used in the study of quasi-periodic systems and in particular in the theory of circle diffeomorphisms.
\subsection{Local aspects: KAM theory}\label{sec:kam}

Let $f(x)=x+\a+v(x)=T_\a(x)+v(x)$ be (the lift of)  a smooth  orientation preserving diffeomorphism of the circle of class $C^r$ (then $v$ is a 1-periodic $C^\infty$ function defined on  $\R$). Let $\rho(f)=\a$, and let us assume  that that $f$ is close to the translation $T_\a$, or equivalently that $v$ is small, in some $C^r$-topology. We look for a conjugating map $h$ of the form $x\mapsto x+w(x)$, with $w$ $1$-periodic and small in some $C^r$-norm, such that $h\circ f= T_{\a}\circ h$, i.e.
$$ x+\a+v(x)+w(x+\a+v(x))=x+\a+w(x).
$$
Up to {\it higher order terms} we must have
$$w(x+\a)-w(x)=-v(x)+O_{2}(v,w)$$
where the term $O_{2}(v,w)$ involves quadratic terms in $v,w$ and their derivatives. 
Conversely if $w$ satisfies
\be w(x+\a)-w(x)=-v(x) \tag{$\cL$} \label{cohomeq}
\ee
then with $h(x)=x+w(x)$ one has a conjugacy 
\be h\circ f=(T_{\a}+v_{2})\circ h \tag{$\mathcal C$} \label{onestepeq}
\ee
where $v_{2}=O_{2}(v,w)$ is quadratic  in $v,w$ and their derivatives.
The { linearized equation} (\ref{cohomeq}) is  a  {\it cohomological} equation which is  traditionally  called the {\it homological equation}. It can be solved by using Fourier series: if $v(x)=\sum_{k\in\Z}\hat v(k)e^{2\pi i kx}$, $w(x)=\sum_{k\in\Z}\hat w(k)e^{2\pi i kx}$ one must have $\hat v(0)=0$ and 
$$\forall k\in\Z^*,\ \hat w(k)=\frac{\hat v(k)}{e^{2\pi k\a}-1}.$$
To overcome the {\it small divisors} problem caused by the possibly small (or zero) denominators in the previous equation, a Diophantine  assumption has to be made on $\a$. If $\a\in \mathcal{CD}(\g,\s)$, then one gets
$$|\hat w(k)|\lesssim \gamma^{-1}|k|^{2+\sigma} |\hat v(k)|.$$
Under this condition (and the assumption that $\hat v(0)=0$) one can get a smooth $w$ solving (\ref{cohomeq}), but with a {\it loss of differentiability}: to get a control on the $C^s$-norm of $w$ one needs for example  a control on the $C^{s'}$-norm of $v$ for $s'>s+3+\sigma$ (this is not optimal). Strictly speaking, in equation (\ref{cohomeq}) $v$ has to be changed in $v-\hat v(0)$ to ensure that $\hat v(0)=0$ and as a consequence in equation (\ref{onestepeq}) $\a$ has to be changed to $\a+\hat v(0)$ but it can be seen that the condition $\rho(f)=\a$ implies $\hat v(0)=O_{2}(v)$ is quadratic in $v$ and its derivative.

To conclude, if $\a$ is Diophantine (in $\mathcal{CD}(\gamma,\sigma)$) one can solve equation (\ref{cohomeq}) (with loss of derivatives) and get (\ref{onestepeq}) with $v_{2}=O_{2}(v,w)$. Since  $\|w\|_{C^s}\lesssim \gamma^{-1}\|v\|_{C^{s+\sigma+3}}$ (for short $w=O_{1}(v)$) one has $v_{2}=O_{2}(v)$ where $O_{2}(v)$ means quadratic  with respect to $v$ and its derivatives.

The next step is  to iterate the preceding step leading to equation (\ref{onestepeq}): this way one gets sequences $v_{n},w_{n}$ with $w_{n}=O_{1}(v_{n})$, $v_{n+1}=O_{2}(v_{n})$ and 
$$  (id+w_{n})\circ\cdots \circ(id+w_{1}) \circ f=(T_{\a}+v_{n})\circ (id+w_{n})\circ\cdots\circ(id+w_{1})
$$
The main goal is to prove that for each fixed $r$,  $\|v_{n}\|_{C^r}$ and $\|w_{n}\|_{C^r}$ go to zero fast enough to ensure that the sequence $(id+w_{n})\circ\cdots\circ(id+w_{1})$ converges in the $C^\infty$-topology to some $h$ and that at the end $h\circ f=T_{\a}\circ h$. This is where the {\it quadratic} convergence of the scheme (Newton iteration scheme) is crucial: {combined with a {\it truncation procedure} it allows }  to overcome the {\it loss of derivatives} phenomenon.

The scheme we have described is the prototypical example of a KAM scheme. It can be used in many situations but its weaknesses are the following: it can only be applied in {\it perturbative situations},  like in Theorem \ref{theo:arnold}; the smallness of the allowed perturbation  is related to the Diophantine condition and, hence,  can hold only for a set of {\it positive Lebesgue measure} of rotation numbers.
\subsection{Global aspects: Linearizability as a compactness result}\label{bounded}

Let's start with an abstract result.
A deep theorem of Gleason-Montgomery-Zippin \cite{MZ} asserts  that a locally compact topological group without small subgroups (which means that there exists a neighborhood of the identity containing no other topological subgroup than the one reduced to the identity) can be endowed with the structure of a Lie group. The case where the group is compact is much easier to prove. The interesting thing is that groups of diffeomorphisms on a compact manifold have the property of being without small subgroups. As a consequence, if $f$ is a diffeomorphism on a compact manifold such that its iterates form a relatively compact set for the $C^r$-topology, $r\geq 1$, the closure of the group of its iterates form a compact abelian Lie group and (its connected component containing the identity) is basically a finite dimensional torus. It is thus natural to expect in this case quasi-periodicity of the diffeomorphism. 
To illustrate this in a more concrete situation,  we observe that  if a smooth diffeomorphism $f$ of the circle is $C^r$-linearizable then its iterates $f^n$, $n\in\Z$ form a relatively compact set for the $C^r$-topology. Conversely, Herman proved \cite{He} that if the iterates of an orientation preserving diffeomorphism $f$ form a relatively compact set in the $C^r$-topology then $f$ is  $C^r$-linearizable. In this case the conjugating map is more or less explicit: define
$$h_{n}=\frac{id+f+\cdots+f^{n-1}}{n}$$ and observe that 
$$h_{n}\circ f\circ h_{n}^{-1}=id+\frac{f^n-id}{n}\circ h_{n}^{-1}.$$
From the definition of the rotation number, the right hand side of the equation converges (uniformly) as $n$ goes to infinity to $T_{\rho(f)}$. On the other hand, since the iterates $f^n$ are bounded in the $C^r$-topology, one can extract a sequence $n_{k}$ such that $h_{n_{k}}$ converges in the $C^{r-1}$-topology to some $C^r$ diffeomorphism $h$ that, as Herman shows, turns out to be actually $C^r$, which implies the $C^r$-linearization.

Compactness criteria are also useful in the analytic case (for holomorphic germs or analytic diffeomorphisms of the circle) but then take the form of topological stability: an orientation preserving analytic diffeomorphism of the circle $f$ is analytically linearizable if the real axis is Lyapunov stable under iterates of (a complex extension) of $f$: to ensure that the iterates of a point under $f$ stay in any given neighborhood of the real axis it is enough to choose this point in a small enough neighborhood of  the real axis. 
 The proof of this geometric compactness criterium follows from the previous construction and Montel's theorem or from an argument of conformal representation.

\subsection{Global aspects: Importance of the geometry in the proof of Herman-Yoccoz Theorem}\label{ss23}

{ For $\a \in (0,1) \setminus \Q$, let $p_n/q_n$ denote the convergents of $\a$.
The topological conjugacy of  Theorem \ref{theo:denjoy} obtained by Denjoy  relies on the following estimate of the growth of the derivatives of $f$ 
\begin{equation}\vert \ln Df^{q_n}(x) \vert \leq {\rm Var}(\ln Df) \quad \forall x\in\T.\label{eq:denjoy}\end{equation}
This estimate in turn relies on the same ordering of the orbits of $f$ and those of $T_\a$, and on Koksma's inequality: {\it for any function $\varphi$ of bounded variation }
$$ |\sum_{j=0}^{q_n-1} \varphi(x+j\a) -  \sum_{j=0}^{q_n-1} \varphi(y+j\a)| \leq {\rm Var}(\varphi)\quad \forall x,y\in\T.$$

 The idea of Herman for obtaining the $C^r$-boundedness of the iterates for a smooth diffeomorphism of Diophantine rotation number was to start by proving a $C^r$-version of Denjoy's estimate for iterates at times $q_n$. Then, he would concatenate the resulting bounds using the adequate arithmetic conditions.

 In the remaining part of this section  \ref{ss23} we briefly describe Yoccoz'  work  \cite{Y-smoothdiffeo} where the same strategy is followed to prove the global linearization theorem in the smooth category for every Diophantine rotation number. A very interesting feature of Yoccoz' work is the  neat and efficient separation between the Denjoy-like estimates that he obtains for a circle diffeomorphism with irrational rotation number, regardless of its arithmetics, and the linearizability consequences that follow if Diophantine properties are thrown in. Moreover, this separation allows to use the same clear cut bounds on the growth of the derivatives, at the special times $q_n$, to study the Liouville case. Namely, the crucial estimate that we will present in Section \ref{sec:estimate}  and {in Section \ref{main} }
 yield the linearizability result of Theorem \ref{A} in the Diophantine case, as well as the density of the linearizable diffeomorphisms of Theorem \ref{theo:yoccoz-fo}  in the Liouville case. This will be a  brief presentation  and we refer to the excellent text \cite{Y-smoothdiffeo} for more details.}
 
We recall the set-up. Let $f$ be a smooth circle diffeomorphism with rotation number $\al$ non-rational. Let $p_n/q_n$ be the convergents of $\a$ and let
 $$\a_n= |q_n \a-p_n|.$$
We denote by $C$ an arbitrary constant that is independent of $n$ and $x\in\T$ (but may depend on $f$ and $r$). 

\subsubsection{Dynamical partitions of the circle and a criterion for $C^1$ conjugacy}\label{ ssA}

For any $x \in \T$, let for $n$ even 
$$I_n(x)=(x,f^{q_n}(x))\quad\mathrm{and}\quad I_{n+1}(x)=(f^{q_{n+1}}(x),x)$$ (for $n$ odd the positions around $x$ are reversed).  From the topological conjugation of $f$ to $T_\a$ (Denjoy) we have that the intervals 
$$I_n(x),\ldots,f^{q_{n+1}-1}(I_n(x)),I_{n+1}(x),\ldots,f^{q_{n}-1}(I_{n+1}(x))$$ 
form a partition of the circle (up to the endpoints of the intervals). 
A criterion of $C^1$-linearization is that these intervals have a comparable size. 

Indeed, let {
$$\beta_n(x)=|I_n(x)|$$
and let 
$$m_n=\min_{x \in \T} \beta_n(x)\le M_n=\max_{x \in \T} \beta_n(x).$$
}
\begin{prop} \label{c1}
If the sequence $(\frac{M_n}{m_n})$ is bounded, then $f$ is $C^1$-conjugated to the rotation $T_\a$.
\end{prop}

\begin{proof} Let $i\in \N$ and $x \in \T$. For any even $n$, there exists $\xi \in I_n(x)$ such that $\beta_n(f^i(x))=Df^i(\xi)\beta_n(x)$. Hence  
$|Df^i(\xi)|\in [C^{-1}, C]$, where $C$ is a bound for $\frac{{M}_n}{{m}_n}$, and as $n$ tends to infinity this implies that $|Df^i(x)|\in [C^{-1},C]$. Since $i$ and $x$ were arbitrary, this shows that the iterates of $f$ are bounded in $C^1$-norm and implies the $C^1$-linearization. \end{proof}

\subsubsection{Distorsion estimates: Cancellations and estimates on the derivatives growth.} \label{sec:estimate}

The crucial estimate on the growth of derivatives is due to Denjoy-like distortion bounds based on formulas involving the chain rule for the Schwarzian derivatives,  as well as on the following simple observation
that 
$$I_n(x),\ldots,f^{q_{n+1}-1}(I_n(x))$$
are disjoint intervals of the circle. 

\medskip 

{\noindent \bf Lemma.} {\it For any $r \in\N^*$ and  any  $1\leq l\le r$ 
$$\sum_{i=0}^{q_{n+1}-1} (Df^i(x))^l \leq C \frac{M_n^{l-1}}{\beta_n(x)^l}\quad \forall x\in \T.$$}

\medskip

\begin{prop}\label{estimate}  For any $r \in \N$ and any  $j\leq q_{n+1}$ we have 
$$|D^r \ln Df^j(x)|\leq C \left[\frac{M_n^{\frac{1}{2}}}{\beta_n(x)} \right]^r\quad \forall x\in \T.$$
\end{prop}

\subsubsection{A priori bounds: Improvement of Denjoy inequality (\ref{eq:denjoy})} \label{denjoy} 

 If we let $J_n(x)=(f^{-q_n}(x),f^{q_n}(x))$, then the intervals 
 $$J_n(x),\ldots,f^{q_{n+1}-1}(J_n(x))$$ 
 cover all the circle.

\begin{prop} $$|\ln Df^{q_n}(x)| \leq C M_n^{1/2}\quad \forall x\in\T.$$
\end{prop}

Take a point $z$ such that $\beta_n(z)={m}_n$. Then $\ln Df^{q_n}(z)=0$  (recall that $m_{n}=\min_{x\in\T}|f^{q_{n}}(x)-x|$). Now,  any point $x\in \T$ can be represented as $f^i(t)$ for some $i\in [0,q_{n+1})$ and $t \in J_n(z)$. Observe that 
$$Df^{q_n+i}(t)=Df^{q_n}(x)Df^i(t)=Df^{i}(f^{q_n}(t))Df^{q_n}(t)$$ 
hence 
\begin{align*} \ln Df^{q_n}(x)&= \ln Df^{i}(f^{q_n}(t))- \ln Df^{i}(t)+ \ln Df^{q_n}(t) \\
&=\ln Df^{i}(f^{q_n}(t))- \ln Df^{i}(t)+ \ln Df^{q_n}(t)-\ln Df^{q_n}(z). \end{align*}
From here, the mean value theorem applied to the two differences above, and the estimate of Proposition \ref{estimate} for $r=1$, 
and the fact that  $|J_n(z)|\leq Cm_n$ (because by the usual Denjoy estimate $Df^{q_n}$ is bounded), yield the improved  Denjoy inequality.
\subsubsection{Controlling the geometry: Relating $m_{n+1}(x)$ to $m_n(x)$} \label{main}

The main ingredient in the proof of $C^1$-linearization is the following estimate.

\begin{prop} \label{prop.main} If $f$ is of class $C^k$ then 
$$\left| \beta_{n+1}(x)-\frac{\a_{n+1}}{\a_n} \beta_n(x) \right| \leq C [M_n^{(k-1)/2}\beta_n(x)+M_n^{1/2} \beta_{n+1}(x)]\quad \forall x \in \T.$$
\end{prop}

{\small \begin{proof} 
For every $x \in \T$,  there exists $y \in [x,f^{q_n}(x)]$,  $z \in [f^{q_{n+1}}(x),x]$ such that 
$$\beta_{n+1}(y)=\frac{\a_{n+1}}{\a_n}\beta_n(z).$$ 
Indeed, this follows from \eqref{eq.rot} and the mean value theorem applied to the two sides of the following identity that is satisfied by the invariance of $\mu$  (the unique invariant probability measure of $f$)
$$\int_x^{f^{q_{n}}(x)} (f^{q_{n+1}}-{\rm Id}) d\mu=\int_x^{f^{q_{n+1}}(x)} (f^{q_{n}}-{\rm Id}) d\mu.$$

Hence, the proposition will follow if one proves 
\begin{equation} \label{111} |\beta_n(z)-\beta_n(x)|\leq CM_n^{1/2} \beta_{n+1}(x) \end{equation} 
and 
\begin{equation} \label{222} |\beta_{n+1}(y)-\beta_{n+1}(x)|\leq C(M_n^{1/2}  \beta_{n+1}(x)+M_n^{(k-1)/2} \beta_n(x)).\end{equation}

Now, \eqref{111}  comes from the improved Denjoy inequality and the fact that $|z-x|\leq  \beta_{n+1}(x)$. 
Finally, \eqref{222} is   proved 
 using an alternative  that we now sketch. 

\medskip

\noindent {\bf Alternative 1.} There exists $ \xi \in \T$ such that  $\beta_{n+1}$ is monotonous on $I_n(\xi)$. Then \eqref{222} will follow if we show that  for any $x\in \T$ and $y \in I_n(x)$  
\begin{equation} \label{333} |\beta_{n+1}(y)/\beta_{n+1}(x)-1|\leq CM_n^{1/2}. \end{equation}
To see this, observe that the improved Denjoy inequality gives  
$$|\beta_{n+1}(f^{q_{n}}(\xi)))-\beta_{n+1}(\xi)|=|\beta_{n}(f^{q_{n+1}}(\xi)))-\beta_{n}(\xi)|\leq C M_n^{1/2} \beta_{n+1}(\xi)$$
yielding by monotonicity \eqref{333} if we replace $x$ and $y$ by  any pair $t,t' \in [f^{-2q_n}(\xi),f^{q_n}(\xi)]$. In particular one can choose the pair $t,t'$ such that 
$f^j(t)=x$, $f^j(t')=y$ for some $j<q_{n+1}$. Now, by the intermediate value theorem 
$$\beta_{n+1}(y)/\beta_{n+1}(x)=(\beta_{n+1}(t')/\beta_{n+1}(t))(Df^j(\theta')/Df^j(\theta))$$
for some $\theta,\theta'\in  [f^{-3q_n}(\xi),f^{q_n}(\xi)]$. {Integrating the estimate of Proposition \eqref{estimate}} for $r=1$ we get that $|Df^j(\theta')/Df^j(\theta)-1|\leq M^{1/2}$, hence \eqref{333}.

\medskip

\noindent {\bf Alternative 2.} For any $\xi \in \T$,  $\beta_{n+1}$ is not monotonous on $I_n(\xi)$. Let $K(x)=[f^{-kq}(x),f^{kq}(x)]$. Observe that by the improved Denjoy estimate we have that $|K(x)|\leq C \beta_n(x)$ and $\beta_n(t)$ is comparable to $\beta_n(x)$ for any $t \in K(x)$. The non-monotonicity hypothesis implies that $Df^{q_{n+1}}-1$ has at least $k$ zeros inside $K(x)$. Hence, an iterative application of Rolle's theorem implies that $D^j \ln Df^{q_{n+1}}$ for $j\leq k-1$, all have zeros inside $K(x)$. The integration of  \eqref{estimate}  then gives for every $t\in K(x)$ that $|\ln Df^{q_{n+1}}(t)|\leq CM^{(k-1)/2}$, which ends the proof of \eqref{222} and thus of Proposition \ref{prop.main}.

\end{proof}
}
An immediate consequence of the Proposition  is the following.

\begin{acorollary} \label{cor:main} If $f$ is of class $C^k$ then 
\begin{align*}
M_{n+1}&\leq M_n \frac{\frac{\a_{n+1}}{\a_n}+CM_n^{(k-1)/2}}{1-CM_n^{1/2}} \\
m_{n+1}&\geq m_n \frac{\frac{\a_{n+1}}{\a_n}-CM_n^{(k-1)/2}}{1+CM_n^{1/2}} 
\end{align*}
\end{acorollary}

\subsubsection{The Diophantine property and $C^1$-linearizability}

 We now assume that $\a$ is Diophantine. It is then a classical fact that there exist  constants $C,\beta>0$ such that $\a_{n+1} \geq C\a_{n}^{1+\beta}$.
Note that by the mean value theorem and \eqref{eq.rot}, we have that 
$$m_n\leq \a_n \leq M_n.$$
As one can see from the first inequality of the corollary of Section \ref{main}, if the term $\frac{\a_{n+1}}{\a_n}$ is dominated by the second one $CM_n^{(k-1)/2}$ then one essentially would have $M_{n+1}\leq M_n^{(k+1)/2}$ which cannot happen too often since  $ \a_{n} \leq M_n$ and $\a_{n+1}\geq C \a_n^{1+\beta}$. Indeed, using this idea it is easy for Yoccoz to show that in fact the first term is dominating all the time after some $n$ and that as an immediate consequence $\frac{M_n}{\a_n}$ is upper bounded. The Corollary then implies that $\frac{\a_n}{m_n}$ is also bounded and finish the $C^1$-conjugacy proof by Proposition \ref{c1}.

\subsubsection{Interpolation inequalities: Bootstrapping the regularity of the conjugacy}

Fix $\gamma\geq0$ and assume that $f$ is $C^{1+\gamma}$ conjugated to $T_\a$. The goal is to show that $f$ is $C^{1+\gamma_1}$ conjugated to $T_\a$ for any $\gamma_1 \in (\gamma,g(\gamma))$ where $g$ is the function {$g(\gamma)=((r-2-\s)+\gamma(1+\s))/(2+\s)$}. Iterating this argument gives that $f$ is $C^{r-1-\s-\eps}$ as claimed in Theorem \ref{A}, since $g$ satisfies $g(\gamma)> \gamma$ on $[0,r-2-\s)$  and {$g(r-2-\s)=r-2-\s$}.  This iteration method takes advantage of {\it a priori} bounds as in Proposition \ref{estimate}  and Hadamard convexity (or interpolation) inequalities that were already used by Herman in \cite{He}.

The  $C^{1+\gamma}$ conjugacy implies 
$$\|f^{q_n}-{\id}-q_n\a\|_{C^{\gamma}} \leq C \|q_n\a\|$$

From Proposition \ref{estimate}  and the $C^1$-conjugacy we get that 
$$\|D^{r-1} \ln Df^{q_n}\|_{C^0} \leq C q_n^{(r-1)/2}.$$
Combining the last two inequalities and using convexity estimates Yoccoz gets a bound on  $\| \ln Df^{q_n}\|_{\gamma_1}$ that after careful concatenation  yields for some $\eps>0$ and for any $n\geq 0$ and any $m\in [0,q_{n+1}/q_n|$ 
$$\| \ln Df^{mq_n}\|_{C^{\gamma_1}} \leq C q_n^{-\eps}$$

Boundedness of the iterates $f^n$ in the $C^{\gamma_1}$ norm then follow since every integer $N$ writes as $N=\sum_{n=0}^S b_n q_n$, $b_n \in [0,q_{n+1}/q_n]\cap\N$.

\subsection{Global dynamics in the Liouville case. Outline of the proof of Theorem \ref{theo:yoccoz-fo}}

When $\a$ is Diophantine, the result follows from Herman-Yoccoz global theorem that established $\cF_\a^\infty=\cO_\a^\infty$ in that case. 

In the Liouville case, the strategy of Yoccoz is to first show that a special class of diffeomorphisms called {\it quasi-rotations} are smoothly linearizable, and then show that Liouville circle diffeomorphisms can be perturbed into quasi-rotations of the same rotation number.

A quasi-rotation is a circle diffeomorphism whose   topological conjugacy to the rotation is affine with  nonzero slope on some interval. 
Another equivalent way of defining them is by supposing that there exists $n\geq 0$ such that  $f^{q_n}$ and $f^{q_{n+1}}$ behave like translations (with angles $\theta \a_n$ and $\theta \a_{n+1}$ for some $\theta>0$) in adequate small neighborhoods of some point $x_0$.

From the latter definition of quasi-rotations it is easy to see that for any $N=\sum_{S\geq s\geq n+1} b_s q_s$ with $S\geq n+1$, $b_s \leq q_{s+1}/q_s$, it holds that $f^{-N}$ is a translation on a small neighborhood of $x_0$. This implies that the negative iterates of $f$ have bounded derivatives of all orders at $x_0$. Since the negative orbit of $x_0$ is dense on the circle (by Denjoy topological conjugacy to $T_\a$) the smooth linearizability of quasi-rotations follows (see Section \ref{bounded}).

To show that a Liouville circle diffeomorphism can be perturbed to a quasi-rotation, estimates similar to the one of Proposition \ref{estimate} are needed. Taking into account the Liouville condition $\a_{n}\leq \a_{n-1}^{(k+1)/2}$ one can prove that 
$$|D^r \ln Df^{q_{n}}(x)|\leq C \frac{M_{n-1}^{\frac{(k-1)}{2}}}{\beta_{n-1}(x)^r}\quad \forall x\in\T.$$ 
The latter inequality, when specialized to the neighborhood of a point $x_0$ where $\beta_{n-1}$ reaches its maximum ($\beta_{n-1}(x_0)=M_{n-1}$), yields 
the required estimate of local flatness of $f^{q_n}$. This allows to perturb $f$ so that the condition of $f^{q_{n}}$ in the definition of the quasi-rotation condition is satisfied. 
Then it is possible to further perturb $f$ to make the condition on $f^{q_{n+1}}$ hold.
 The fact that 
 the first part of
  the condition on $f^{q_{n}}$ holds is used to prove the local flatness of $f^{q_{n+1}}$ required for the second perturbation since it is not possible anymore to rely on the Liouville property between $\a_{n+2}$ and $\a_{n+1}$ (that may not hold).

\subsection{Global aspects: Renormalization. The analytic case.}
The strategy of Yoccoz to analyze linearization in the analytic case follows a different path from the smooth case, namely  {\it renormalization}. The idea is that some well chosen {\it iterates} of a dynamical system can be {\it rescaled} yielding a new dynamical system which is closer to a {\it universal} situation.
Dynamical systems that can be renormalized are {\it a priori} rare.  One has first  to identify a fundamental domain with a not too complicated  geometry and then to  provide estimates for the first return map in this fundamental domain. These estimates  usually rely on arguments of geometric nature. In the case of analytic diffeomorphisms of  the circle the choice of the fundamental domain is quite canonical: let's still denote by $f$ a holomorphic and univalent  extension of (a lift of) $f$ on an neighborhood of the circle $\T$ and consider in this complex neighborhood a real symmetric vertical segment $L$. A choice of a fundamental domain is to consider a domain $U$ (with a deformed rectangular shape)  delimited vertically  by $L$ and $f(L)$ and horizontally by straight lines joining the boundaries of $L$ and $f(L)$. A first difficulty is to prove that, up to considering a thiner fundamental domain $R$ in $U$, first returns of points of  $R$ in $U$ are well defined. Now comes the {\it rescaling} (or the {\it normalization}):  by gluing $L$ and $f(L)$ the domain $U$ becomes an abstract Riemann surface $\hat U$ which is topologically  an annulus. The Uniformization Theorem tells us that $\hat U$ is conformally equivalent to an annulus of (horizontal) length 1. The conformal equivalence transports the domain $R$ and the first return map defined on $R$ into an annulus $\hat R$ and a holomorphic univalent map $F:\hat R\to\hat U$. The gain in this procedure is that if the (vertical) width of the initial cylinder $U$ is of order 1 and if its (horizontal) length is small  (more or less it compares to  $\|f-id\|_{C^0}$) by conformality the width of $\hat U$ (the horizontal length of which is  now of order 1)  is of the order of the inverse of the initial horizontal length. If the width of $\hat R$ is of the same order of magnitude then one ends up with a {\it univalent} map $F$ defined on a annulus of horizontal length one and large vertical width: but by Gr\"otzsch's theorem this forces the univalent map to be close to a translation (the universal model)  on a possibly smaller domain. There are several difficulties in this construction. First we should assume in order to properly define the fundamental domain  that $f$ is already close (in some complex strip) to a translation: considering iterates of the form $f^{q_{n}}$  (and a partial renormalization)
 this is a case to which one can reduce by using the $C^k$-theory of the Herman-Yoccoz theorem, an analysis that is done on the real axis only.  Once this is done, a control in the vertical (complex)  direction of the iterates of $f$  is necessary to control the width  of $R$ (resp. $\hat R$) and hence  the first returns of $f$ (resp. $F$) in $U$ (resp. $\hat U$): Yoccoz proves to that end  a {\it complex version} of Denjoy theorem.  The renormalization  procedure that associates $F$ to $f$ is then iterated yielding a sequence $(F_{n})_{n}$ of univalent maps defined on annuli of horizontal length 1 and larger and larger vertical width. The  preceding analysis then shows that when $\a$ satisfies an adequate arithmetic condition the domain of definitions of the renormalized maps (when viewed in the initial coordinates of the first step) do not shrink to zero, thus providing the proof of Lyapunov stability of the initial map and hence its linearizability (see Section \ref{bounded}). The new feature of this renormalization scheme compared to the one Yoccoz had designed ten years earlier  for holomorphic germs having elliptic fixed points \cite{Y-germs}  is that in the case of analytic circle diffeomorphisms  two different regimes appear: one where the nonlinearity $v_{n}$ of $F_{n}=T_{\a_{n}}+v_{n}$ is small compared to its translation part $T_{\a_{n}}$ ($\a_{n}=G^n(\a)$) and one where this nonlinearity  dominates. In the first case, the logarithm of the  nonlinearity of $F_{n+1}$ is of the order of $-\a_{n}^{-1}$ while in the second case it is of the order of $-\|v_{n}\|^{-1}$ (and is thus not anymore related to the arithmetics of $\a$ in a direct way).    This is at the origin of the fact that the needed arithmetic assumptions to ensure the linearization of an analytic circle  diffeomorphism is different whether we are in the local (condition $\cB$)  or the global case (condition $\cH$); this also explains why condition $\cH$ is more difficult to describe than condition $\cB$.

An important outcome of this renormalization paradigm is that it enabled Yoccoz to prove that condition $\cH$ is stronger than condition $\cB$. By {\it inverting} his renormalization construction he was able to construct analytic circle diffeomorphism satisfying condition $\cB$ which are not analytically linearizable; {\it cf.} Theorem \ref{theo:yoccoz-analytic-global}. This is a very nice construction that we shall not describe here; the interested reader is referred to  the excellent text \cite{Y-analyticdiffeo}.

Technically, it is more convenient to implement the preceding procedures  by using {\it commuting pairs} (instead of first return maps) since the link with the arithmetics of $\a$ is more transparent; this is a point we shall not develop further.

\end{document}